\documentclass{amsart}
\usepackage{amsmath}
\usepackage{amssymb}
\usepackage{amsfonts}

\newtheorem{theorem}{Theorem}
\theoremstyle{plain}

\newtheorem{corollary}{Corollary}

\newtheorem{lemma}{Lemma}

\newtheorem{proposition}{Proposition}
\newtheorem{remark}{Remark}

\numberwithin{equation}{section}
\input{tcilatex}

\begin{document}
\title[SOME INTEGRAL INEQUALITIES FOR $s-$CONVEX FUNCTIONS]{SOME OSTROWSKI'S
TYPE INEQUALITIES FOR FUNCTIONS WHOSE SECOND DERIVATIVES ARE $s-$CONVEX IN
THE SECOND SENSE AND APPLICATIONS}
\author{ERHAN SET$^{\ast ,\blacksquare }$}
\address{$^{\blacksquare }$Atat\"{u}rk University, K.K. Education Faculty,
Department of Mathematics, 25240, Campus, Erzurum, Turkey}
\email{erhanset@yahoo.com}
\thanks{$^{\ast }$corresponding author}
\author{Mehmet Zeki SARIKAYA$^{\blacktriangledown }$}
\address{$^{\blacktriangledown }$Department of Mathematics, Faculty of
Science and Arts, D\"{u}zce University, D\"{u}zce, Turkey}
\email{sarikayamz@gmail.com}
\author{M. Emin Ozdemir$^{\blacksquare }$}
\address{$^{\blacksquare }$Atat\"{u}rk University, K.K. Education Faculty,
Department of Mathematics, 25240, Campus, Erzurum, Turkey}
\email{emos@atauni.edu.tr}

\begin{abstract}
Some new inequalities of Ostrowski type for twice differentiable mappings
whose derivatives in absolute value are $s-$convex in the second sense are
given.Applications for special means are also provided.
\end{abstract}

\maketitle

\section{INTRODUCTION}

In 1938, Ostrowski proved the following integral inequality \cite{AO}:

\begin{theorem}
Let $f:I\subseteq \mathbb{R}\rightarrow \mathbb{R}$ be a differentiable
mapping on $\left( a,b\right) $ whose derivative $f^{\prime }:\left(
a,b\right) \rightarrow \mathbb{R}$ is bounded on $\left( a,b\right) $, i.e., 
$\left\Vert f^{\prime }\right\Vert _{\infty }=\underset{t\in \left(
a,b\right) }{\sup }\left\vert f^{\prime }(t)\right\vert <\infty .$ Then, the
inequality holds:%
\begin{equation*}
\left\vert f(x)-\frac{1}{b-a}\int_{a}^{b}f(t)dt\right\vert \leq \left[ \frac{%
1}{4}+\frac{\left( x-\frac{a+b}{2}\right) ^{2}}{\left( b-a\right) ^{2}}%
\right] \left( b-a\right) \left\Vert f^{\prime }\right\Vert _{\infty }
\end{equation*}%
for all $x\in \left[ a,b\right] .$ The constant $\frac{1}{4}$ is sharp in
the sense that it cannot be replaced by a smaller one.
\end{theorem}

For some applications of Ostrowski's inequality see (\cite{AD}-\cite{DW})
and for recent results and generalizations concerning Ostrowski's inequality
see (\cite{AD}-\cite{SSO}).

The class of $s-$convexity in the second sense is defined in the following
way \cite{Breckner}: a function $f:[0,\infty )\rightarrow \mathbb{R}$ is
said to be $s-$convex in the second sense if 
\begin{equation*}
f(tx+(1-t)y)\leq t^{s}f(x)+(1-t)^{s}f(y)
\end{equation*}%
for all $x,y\in \lbrack 0,\infty ),\;t\in \lbrack 0,1]$ and some fixed $s\in
(0,1].$This class is usually denoted by $K_{s}^{2}.$

In \cite{dragomir1}, Dragomir and Fitzpatrick proved te Hadamard's
inequality for $s-$convex functions in the second sense:

\begin{theorem}
Suppose that $f:[0,\infty )\rightarrow \lbrack 0,\infty )$ is an $s$-convex
function in the second sense, where $s\in (0,1),$ and let $a,b\in \lbrack
0,\infty ),$ $a<b.$ If $f\in L^{1}(\left[ a,b\right] ),$ then the following
inequalities hold:%
\begin{equation}
2^{s-1}f(\frac{a+b}{2})\leq \frac{1}{b-a}\int\limits_{a}^{b}f(x)dx\leq \frac{%
f(a)+f(b)}{s+1}.  \label{e.1.1}
\end{equation}
\end{theorem}

The constant $k=\frac{1}{s+1}$ is the best possible in the second inequality
in (\ref{e.1.1}).

In \cite{CDR}, Cerone et.al. proved the following inequalities of Ostrowski
type and Hadamard type, respectively.

\begin{theorem}
Let $f:[a,b]\rightarrow \mathbb{R}$ be a twice differentiable mapping on $%
(a,b)$ and $f^{\prime \prime }:(a,b)\rightarrow \mathbb{R}$ is bounded,
i.e., $\left\Vert f^{\prime \prime }\right\Vert _{\infty }=\underset{t\in
\left( a,b\right) }{\sup }\left\vert f^{\prime \prime }(t)\right\vert
<\infty .$ Then we have the inequality:%
\begin{eqnarray}
&&\left\vert f(x)-\frac{1}{b-a}\int_{a}^{b}f(t)dt-\left( x-\frac{a+b}{2}%
\right) f^{\prime }(x)\right\vert  \label{e.1.2} \\
&\leq &\left[ \frac{1}{24}(b-a)^{2}+\frac{1}{2}\left( x-\frac{a+b}{2}\right)
^{2}\right] \left\Vert f^{\prime \prime }\right\Vert _{\infty }  \notag \\
&\leq &\frac{(b-a)^{2}}{6}\left\Vert f^{\prime \prime }\right\Vert _{\infty }
\notag
\end{eqnarray}%
for all $x\in \lbrack a,b].$
\end{theorem}

\begin{corollary}
Under the above assumptions, we have the mid-point inequality: 
\begin{equation}
\left\vert f(\frac{a+b}{2})-\frac{1}{b-a}\int_{a}^{b}f(x)dx\right\vert \leq 
\frac{(b-a)^{2}}{24}\left\Vert f^{\prime \prime }\right\Vert _{\infty }
\label{e.1.3}
\end{equation}
\end{corollary}

In this article, we establish new Ostrowski's type inequalities for $s-$%
convex functions in the second sense and using this results we note some
applications to special means.

\section{Main Results}

In order to establish our main results we need the following Lemma.

\begin{lemma}
\label{L.1} Let $f:I\subseteq \mathbb{R}\rightarrow \mathbb{R}$ be a twice
differentiable function on $I^{\circ }$ with $f^{\prime \prime }\in
L_{1}[a,b],$ then 
\begin{eqnarray}
&&\frac{1}{b-a}\int_{a}^{b}f(u)du-f(x)+\left( x-\frac{a+b}{2}\right)
f^{\prime }(x)  \label{e.2.1} \\
&=&\frac{(x-a)^{3}}{2(b-a)}\int_{0}^{1}t^{2}f^{\prime \prime }(tx+(1-t)a)dt+%
\frac{(b-x)^{3}}{2(b-a)}\int_{0}^{1}t^{2}f^{\prime \prime }(tx+(1-t)b)dt 
\notag
\end{eqnarray}
\end{lemma}

\begin{proof}
By integration by parts, we have the following identity 
\begin{eqnarray}
I_{1} &=&\int_{0}^{1}t^{2}f^{\prime \prime }(tx+(1-t)a)dt  \label{e.2.2} \\
&=&\left. \frac{t^{2}}{(x-a)}f^{\prime }(tx+(1-t)a)\right\vert _{0}^{1}-%
\frac{2}{x-a}\int_{0}^{1}tf^{\prime }(tx+(1-t)a)dt  \notag \\
&=&\frac{f^{\prime }(x)}{(x-a)}-\frac{2}{x-a}\left[ \left. \frac{t}{(x-a)}%
f(tx+(1-t)a)\right\vert _{0}^{1}\right.  \notag \\
&&\left. -\frac{1}{x-a}\int_{0}^{1}f(tx+(1-t)a)dt\right]  \notag \\
&=&\frac{f^{\prime }(x)}{(x-a)}-\frac{2f(x)}{(x-a)^{2}}+\frac{2}{(x-a)^{2}}%
\int_{0}^{1}f(tx+(1-t)a)dt  \notag
\end{eqnarray}

Using the change of the variable $u=tx+(1-t)a$ for $\;t\in \lbrack 0,1]$ and
by multiplying the both sides (\ref{e.2.2}) by $\frac{(x-a)^{3}}{2(b-a)},$
we obtain 
\begin{eqnarray}
&&\frac{(x-a)^{3}}{2(b-a)}\int_{0}^{1}t^{2}f^{\prime \prime }(tx+(1-t)a)dt
\label{e.2.3} \\
&=&\frac{(x-a)^{2}f^{\prime }(x)}{2(b-a)}-\frac{(x-a)f(x)}{b-a}+\frac{1}{b-a}%
\int_{a}^{x}f(u)du  \notag
\end{eqnarray}%
Similarly, we observe that%
\begin{eqnarray}
&&\frac{(b-x)^{3}}{2(b-a)}\int_{0}^{1}t^{2}f^{\prime \prime }(tx+(1-t)b)dt
\label{e.2.4} \\
&=&-\frac{(b-x)^{2}f^{\prime }(x)}{2(b-a)}-\frac{(b-x)f(x)}{b-a}+\frac{1}{b-a%
}\int_{x}^{b}f(u)du  \notag
\end{eqnarray}%
Thus, adding (\ref{e.2.3}) and (\ref{e.2.4}) we get the required identity (%
\ref{e.2.1}).
\end{proof}

The following result may be stated:

\begin{theorem}
\label{teo1} Let $f:I\subset \lbrack 0,\infty )\rightarrow \mathbb{R}$ be a
twice differentiable function on $I^{\circ }$ such that $f^{\prime \prime
}\in L_{1}[a,b]$ where $a,b\in I$ with $a<b.$ If $\left\vert f^{\prime
\prime }\right\vert $ is $s-$convex in the second sense on $[a,b]$ for some
fixed $s\in (0,1],$ then the following inequality holds:%
\begin{eqnarray}
&&\left\vert \frac{1}{b-a}\int_{a}^{b}f(u)du-f(x)+\left( x-\frac{a+b}{2}%
\right) f^{\prime }(x)\right\vert  \label{e.2.5} \\
&\leq &\frac{1}{2(b-a)}\left\{ \left[ \frac{\left\vert f^{\prime \prime
}(x)\right\vert }{s+3}+\frac{2\left\vert f^{\prime \prime }(a)\right\vert }{%
(s+1)(s+2)(s+3)}\right] (x-a)^{3}\right.  \notag \\
&&\text{ \ \ \ \ \ \ \ \ \ \ }\left. +\left[ \frac{\left\vert f^{\prime
\prime }(x)\right\vert }{s+3}+\frac{2\left\vert f^{\prime \prime
}(b)\right\vert }{(s+1)(s+2)(s+3)}\right] (b-x)^{3}\right\}  \notag
\end{eqnarray}%
for each $x\in \lbrack a,b].$
\end{theorem}

\begin{proof}
From Lemma \ref{L.1} and since $\left\vert f^{\prime \prime }\right\vert $
is $s-$convex, then we have%
\begin{eqnarray*}
&&\left\vert \frac{1}{b-a}\int_{a}^{b}f(u)du-f(x)+\left( x-\frac{a+b}{2}%
\right) f^{\prime }(x)\right\vert \\
&\leq &\frac{(x-a)^{3}}{2(b-a)}\int_{0}^{1}t^{2}\left\vert f^{\prime \prime
}(tx+(1-t)a)\right\vert dt+\frac{(b-x)^{3}}{2(b-a)}\int_{0}^{1}t^{2}\left%
\vert f^{\prime \prime }(tx+(1-t)b)\right\vert dt \\
&\leq &\frac{(x-a)^{3}}{2(b-a)}\int_{0}^{1}t^{2}\left[ t^{s}\left\vert
f^{\prime \prime }(x)\right\vert +(1-t)^{s}\left\vert f^{\prime \prime
}(a)\right\vert \right] dt \\
&&+\frac{(b-x)^{3}}{2(b-a)}\int_{0}^{1}t^{2}\left[ t^{s}\left\vert f^{\prime
\prime }(x)\right\vert +(1-t)^{s}\left\vert f^{\prime \prime }(b)\right\vert %
\right] dt \\
&=&\frac{(x-a)^{3}}{2(b-a)}\int_{0}^{1}\left( t^{s+2}\left\vert f^{\prime
\prime }(x)\right\vert +t^{2}(1-t)^{s}\left\vert f^{\prime \prime
}(a)\right\vert \right) dt \\
&&+\frac{(b-x)^{3}}{2(b-a)}\int_{0}^{1}\left( t^{s+2}\left\vert f^{\prime
\prime }(x)\right\vert +t^{2}(1-t)^{s}\left\vert f^{\prime \prime
}(b)\right\vert \right) dt
\end{eqnarray*}
\begin{eqnarray*}
&=&\frac{(x-a)^{3}}{2(b-a)}\left[ \frac{\left\vert f^{\prime \prime
}(x)\right\vert }{s+3}+\frac{2\left\vert f^{\prime \prime }(a)\right\vert }{%
(s+1)(s+2)(s+3)}\right] \\
&&+\frac{(b-x)^{3}}{2(b-a)}\left[ \frac{\left\vert f^{\prime \prime
}(x)\right\vert }{s+3}+\frac{2\left\vert f^{\prime \prime }(b)\right\vert }{%
(s+1)(s+2)(s+3)}\right] \\
&=&\frac{1}{2(b-a)}\left\{ \left[ \frac{\left\vert f^{\prime \prime
}(x)\right\vert }{s+3}+\frac{2\left\vert f^{\prime \prime }(a)\right\vert }{%
(s+1)(s+2)(s+3)}\right] (x-a)^{3}\right. \\
&&\text{\ \ \ \ \ }\left. +\left[ \frac{\left\vert f^{\prime \prime
}(x)\right\vert }{s+3}+\frac{2\left\vert f^{\prime \prime }(b)\right\vert }{%
(s+1)(s+2)(s+3)}\right] (b-x)^{3}\right\}
\end{eqnarray*}%
where we have used the fact that 
\begin{equation*}
\int_{0}^{1}t^{s+2}dt=\frac{1}{s+3}\text{ \ \ \ \ and \ \ \ \ }%
\int_{0}^{1}t^{2}(1-t)^{s}dt\text{\ }=\frac{2}{(s+1)(s+2)(s+3)}.\text{\ }
\end{equation*}%
This completes the proof.
\end{proof}

\begin{corollary}
\label{cor1} We choose $\left\vert f^{\prime \prime }(x)\right\vert \leq
M,M>0$ in Theorem \ref{teo1}, then we have 
\begin{eqnarray}
&&\left\vert \frac{1}{b-a}\int_{a}^{b}f(u)du-f(x)+\left( x-\frac{a+b}{2}%
\right) f^{\prime }(x)\right\vert  \label{e.2.6} \\
&\leq &3M\left( \frac{s^{2}+3s+4}{(s+1)(s+2)(s+3)}\right) \left[ \frac{1}{24}%
(b-a)^{2}+\frac{1}{2}\left( x-\frac{a+b}{2}\right) ^{2}\right]  \notag \\
&\leq &M\frac{(b-a)^{2}}{2}\left( \frac{s^{2}+3s+4}{(s+1)(s+2)(s+3)}\right) .
\notag
\end{eqnarray}%
Here, by simple computation shows that%
\begin{equation*}
(x-a)^{3}+(b-x)^{3}=(b-a)\left[ \frac{(b-a)^{2}}{4}+3\left( x-\frac{a+b}{2}%
\right) ^{2}\right] .
\end{equation*}
\end{corollary}

\begin{remark}
\label{R1} If in Corollary \ref{cor1} we choose $s=1$, then we recapture the
inequality (\ref{e.1.2}).
\end{remark}

\begin{corollary}
\label{cor2} If in Corollary \ref{cor1} we choose $x=\frac{a+b}{2},$ then we
get the mid-point inequality%
\begin{equation*}
\left\vert \frac{1}{b-a}\int_{a}^{b}f(u)du-f(\frac{a+b}{2})\right\vert \leq M%
\frac{(b-a)^{2}}{2}\left( \frac{s^{2}+3s+4}{(s+1)(s+2)(s+3)}\right) .
\end{equation*}

\begin{theorem}
\label{teo2} Let $f:I\subset \lbrack 0,\infty )\rightarrow \mathbb{R}$ be a
twice differentiable function on $I^{\circ }$ such that $f^{\prime \prime
}\in L_{1}[a,b]$ where $a,b\in I$ with $a<b.$ If $\left\vert f^{\prime
\prime }\right\vert ^{q}$ is $s-$convex in the second sense on $[a,b]$ for
some fixed $s\in (0,1],$ $p,q>1$ and $\frac{1}{p}+\frac{1}{q}=1,$ then the
following inequality holds:%
\begin{eqnarray}
&&\left\vert \frac{1}{b-a}\int_{a}^{b}f(u)du-f(x)+\left( x-\frac{a+b}{2}%
\right) f^{\prime }(x)\right\vert  \label{e.2.7} \\
&\leq &\frac{(x-a)^{3}}{2(b-a)}\left( \frac{1}{2p+1}\right) ^{\frac{1}{p}%
}\left( \frac{\left\vert f^{\prime \prime }(x)\right\vert ^{q}+\left\vert
f^{\prime \prime }(a)\right\vert ^{q}}{s+1}\right) ^{\frac{1}{q}}  \notag \\
&&+\frac{(b-x)^{3}}{2(b-a)}\left( \frac{1}{2p+1}\right) ^{\frac{1}{p}}\left( 
\frac{\left\vert f^{\prime \prime }(x)\right\vert ^{q}+\left\vert f^{\prime
\prime }(b)\right\vert ^{q}}{s+1}\right) ^{\frac{1}{q}}  \notag
\end{eqnarray}%
for each $x\in \lbrack a,b].$
\end{theorem}
\end{corollary}

\begin{proof}
Suppose that $p>1.$ From Lemma \ref{L.1} and using the H\"{o}lder
inequality, we have%
\begin{eqnarray*}
&&\left\vert \frac{1}{b-a}\int_{a}^{b}f(u)du-f(x)+\left( x-\frac{a+b}{2}%
\right) f^{\prime }(x)\right\vert \\
&\leq &\frac{(x-a)^{3}}{2(b-a)}\int_{0}^{1}t^{2}\left\vert f^{\prime \prime
}(tx+(1-t)a)\right\vert dt+\frac{(b-x)^{3}}{2(b-a)}\int_{0}^{1}t^{2}\left%
\vert f^{\prime \prime }(tx+(1-t)b)\right\vert dt \\
&\leq &\frac{(x-a)^{3}}{2(b-a)}\left( \int_{0}^{1}t^{2p}dt\right) ^{\frac{1}{%
p}}\left( \int_{0}^{1}\left\vert f^{\prime \prime }(tx+(1-t)a)\right\vert
^{q}dt\right) ^{\frac{1}{q}} \\
&&+\frac{(b-x)^{3}}{2(b-a)}\left( \int_{0}^{1}t^{2p}dt\right) ^{\frac{1}{p}%
}\left( \int_{0}^{1}\left\vert f^{\prime \prime }(tx+(1-t)b)\right\vert
^{q}dt\right) ^{\frac{1}{q}}.
\end{eqnarray*}%
Since $\left\vert f^{\prime \prime }\right\vert ^{q}$ is $s-$convex in the
second sense, then we have%
\begin{eqnarray*}
\int_{0}^{1}\left\vert f^{\prime \prime }(tx+(1-)a)\right\vert ^{q}dt &\leq
&\int_{0}^{1}\left[ t^{s}\left\vert f^{\prime \prime }(x)\right\vert
^{q}+(1-t)^{s}\left\vert f^{\prime \prime }(a)\right\vert ^{q}\right] dt \\
&=&\frac{\left\vert f^{\prime \prime }(x)\right\vert ^{q}+\left\vert
f^{\prime \prime }(a)\right\vert ^{q}}{s+1}
\end{eqnarray*}%
and 
\begin{eqnarray*}
\int_{0}^{1}\left\vert f^{\prime \prime }(tx+(1-t)b)\right\vert ^{q}dt &\leq
&\int_{0}^{1}\left[ t^{s}\left\vert f^{\prime \prime }(x)\right\vert
^{q}+(1-t)^{s}\left\vert f^{\prime \prime }(b)\right\vert ^{q}\right] dt \\
&=&\frac{\left\vert f^{\prime \prime }(x)\right\vert ^{q}+\left\vert
f^{\prime \prime }(b)\right\vert ^{q}}{s+1}.
\end{eqnarray*}%
Therefore, we have 
\begin{eqnarray*}
&&\left\vert \frac{1}{b-a}\int_{a}^{b}f(u)du-f(x)+\left( x-\frac{a+b}{2}%
\right) f^{\prime }(x)\right\vert \\
&\leq &\frac{(x-a)^{3}}{2(b-a)}\left( \frac{1}{2p+1}\right) ^{\frac{1}{p}%
}\left( \frac{\left\vert f^{\prime \prime }(x)\right\vert ^{q}+\left\vert
f^{\prime \prime }(a)\right\vert ^{q}}{s+1}\right) ^{\frac{1}{q}} \\
&&+\frac{(b-x)^{3}}{2(b-a)}\left( \frac{1}{2p+1}\right) ^{\frac{1}{p}}\left( 
\frac{\left\vert f^{\prime \prime }(x)\right\vert ^{q}+\left\vert f^{\prime
\prime }(b)\right\vert ^{q}}{s+1}\right) ^{\frac{1}{q}}
\end{eqnarray*}%
where $\frac{1}{p}+\frac{1}{q}=1,$ which is required.
\end{proof}

\begin{corollary}
\label{cor3}Under the above assumptions we have the following inequality:%
\begin{eqnarray}
&&\left\vert \frac{1}{b-a}\int_{a}^{b}f(u)du-f(x)+\left( x-\frac{a+b}{2}%
\right) f^{\prime }(x)\right\vert  \label{e.2.8} \\
&\leq &\frac{3M}{\left( 2p+1\right) ^{\frac{1}{p}}}\left( \frac{2}{s+1}%
\right) ^{\frac{1}{q}}\left[ \frac{(b-a)^{2}}{24}+\frac{1}{2}\left( x-\frac{%
a+b}{2}\right) ^{2}\right] .  \notag
\end{eqnarray}%
This follows by Theorem \ref{teo2}, choosing $\left\vert f^{\prime \prime
}(x)\right\vert \leq M,\;M>0.$
\end{corollary}

\begin{corollary}
\label{cor4} With the assumptions in Corollary \ref{cor3}, one has the
mid-point inequality:%
\begin{equation*}
\left\vert \frac{1}{b-a}\int_{a}^{b}f(u)du-f\left( \frac{a+b}{2}\right)
\right\vert \leq \frac{(b-a)^{2}}{8\left( 2p+1\right) ^{\frac{1}{p}}}\left( 
\frac{2}{s+1}\right) ^{\frac{1}{q}}M.
\end{equation*}%
This follows by Corollary \ref{cor3}, choosing $x=\frac{a+b}{2}.$
\end{corollary}

\begin{corollary}
\label{cor5} With the assumptions in Corollary \ref{cor3}, one has the
following perturbed trapezoid like inequality:%
\begin{multline*}
\left\vert \int_{a}^{b}f(u)du-\frac{\left( b-a\right) }{2}\left[ f\left(
a\right) +f\left( b\right) \right] +\frac{\left( b-a\right) ^{2}}{4}\left(
f^{\prime }\left( b\right) -f^{\prime }\left( a\right) \right) \right\vert \\
\leq \frac{(b-a)^{3}}{2\left( 2p+1\right) ^{\frac{1}{p}}}\left( \frac{2}{s+1}%
\right) ^{\frac{1}{q}}M.
\end{multline*}%
This follows using Corollary \ref{cor3} with $x=a$, $x=b$, adding the
results and using the triangle inequality for the modulus.
\end{corollary}

\begin{theorem}
\label{teo3} Let $f:I\subset \lbrack 0,\infty )\rightarrow \mathbb{R}$ be a
twice differentiable function on $I^{\circ }$ such that $f^{\prime \prime
}\in L_{1}[a,b]$ where $a,b\in I$ with $a<b.$ If $\left\vert f^{\prime
\prime }\right\vert ^{q}$ is $s-$convex in the second sense on $[a,b]$ for
some fixed $s\in (0,1]$ and $q\geq 1,$ then the following inequality holds:%
\begin{eqnarray}
&&\left\vert \frac{1}{b-a}\int_{a}^{b}f(u)du-f(x)+\left( x-\frac{a+b}{2}%
\right) f^{\prime }(x)\right\vert \\
&\leq &\frac{(x-a)^{3}}{2(b-a)}\left( \frac{1}{3}\right) ^{1-\frac{1}{q}%
}\left( \frac{\left\vert f^{\prime \prime }(x)\right\vert ^{q}}{s+3}+\frac{%
2\left\vert f^{\prime \prime }(a)\right\vert ^{q}}{\left( s+1\right) \left(
s+2\right) \left( s+3\right) }\right) ^{\frac{1}{q}}  \notag \\
&&+\frac{(b-x)^{3}}{2(b-a)}\left( \frac{1}{3}\right) ^{1-\frac{1}{q}}\left( 
\frac{\left\vert f^{\prime \prime }(x)\right\vert ^{q}}{s+3}+\frac{%
2\left\vert f^{\prime \prime }(b)\right\vert ^{q}}{\left( s+1\right) \left(
s+2\right) \left( s+3\right) }\right) ^{\frac{1}{q}}  \notag
\end{eqnarray}%
for each $x\in \lbrack a,b].$
\end{theorem}

\begin{proof}
Suppose that $q\geq 1.$ From Lemma \ref{L.1} and using the well known power
mean inequality, we have%
\begin{eqnarray*}
&&\left\vert \frac{1}{b-a}\int_{a}^{b}f(u)du-f(x)+\left( x-\frac{a+b}{2}%
\right) f^{\prime }(x)\right\vert \\
&\leq &\frac{(x-a)^{3}}{2(b-a)}\int_{0}^{1}t^{2}\left\vert f^{\prime \prime
}(tx+(1-t)a)\right\vert dt+\frac{(b-x)^{3}}{2(b-a)}\int_{0}^{1}t^{2}\left%
\vert f^{\prime \prime }(tx+(1-t)b)\right\vert dt \\
&\leq &\frac{(x-a)^{3}}{2(b-a)}\left( \int_{0}^{1}t^{2}dt\right) ^{1-\frac{1%
}{q}}\left( \int_{0}^{1}t^{2}\left\vert f^{\prime \prime
}(tx+(1-t)a)\right\vert ^{q}dt\right) ^{\frac{1}{q}} \\
&&+\frac{(b-x)^{3}}{2(b-a)}\left( \int_{0}^{1}t^{2}dt\right) ^{1-\frac{1}{q}%
}\left( \int_{0}^{1}t^{2}\left\vert f^{\prime \prime }(tx+(1-t)b)\right\vert
^{q}dt\right) ^{\frac{1}{q}}.
\end{eqnarray*}%
Since $\left\vert f^{\prime \prime }\right\vert ^{q}$ is $s-$convex in the
second sense, we have%
\begin{eqnarray*}
\int_{0}^{1}t^{2}\left\vert f^{\prime \prime }(tx+(1-t)a)\right\vert ^{q}dt
&\leq &\int_{0}^{1}\left[ t^{s+2}\left\vert f^{\prime \prime }(x)\right\vert
^{q}+t^{2}(1-t)^{s}\left\vert f^{\prime \prime }(a)\right\vert ^{q}\right] dt
\\
&=&\frac{\left\vert f^{\prime \prime }(x)\right\vert ^{q}}{s+3}+\frac{%
2\left\vert f^{\prime \prime }(a)\right\vert ^{q}}{\left( s+1\right) \left(
s+2\right) \left( s+3\right) }
\end{eqnarray*}%
and 
\begin{eqnarray*}
\int_{0}^{1}t^{2}\left\vert f^{\prime \prime }(tx+(1-t)b)\right\vert ^{q}dt
&\leq &\int_{0}^{1}\left[ t^{s+2}\left\vert f^{\prime \prime }(x)\right\vert
^{q}+t^{2}(1-t)^{s}\left\vert f^{\prime \prime }(b)\right\vert ^{q}\right] dt
\\
&=&\frac{\left\vert f^{\prime \prime }(x)\right\vert ^{q}}{s+3}+\frac{%
2\left\vert f^{\prime \prime }(b)\right\vert ^{q}}{\left( s+1\right) \left(
s+2\right) \left( s+3\right) }.
\end{eqnarray*}%
Therefore, we have 
\begin{eqnarray*}
&&\left\vert \frac{1}{b-a}\int_{a}^{b}f(u)du-f(x)+\left( x-\frac{a+b}{2}%
\right) f^{\prime }(x)\right\vert \\
&\leq &\frac{(x-a)^{3}}{2(b-a)}\left( \frac{1}{3}\right) ^{1-\frac{1}{q}%
}\left( \frac{\left\vert f^{\prime \prime }(x)\right\vert ^{q}}{s+3}+\frac{%
2\left\vert f^{\prime \prime }(a)\right\vert ^{q}}{\left( s+1\right) \left(
s+2\right) \left( s+3\right) }\right) ^{\frac{1}{q}} \\
&&+\frac{(b-x)^{3}}{2(b-a)}\left( \frac{1}{3}\right) ^{1-\frac{1}{q}}\left( 
\frac{\left\vert f^{\prime \prime }(x)\right\vert ^{q}}{s+3}+\frac{%
2\left\vert f^{\prime \prime }(b)\right\vert ^{q}}{\left( s+1\right) \left(
s+2\right) \left( s+3\right) }\right) ^{\frac{1}{q}}.
\end{eqnarray*}
\end{proof}

\begin{corollary}
\label{cor6} Under the above assumptions we have the following inequality%
\begin{eqnarray*}
&&\left\vert \frac{1}{b-a}\int_{a}^{b}f(u)du-f(x)+\left( x-\frac{a+b}{2}%
\right) f^{\prime }(x)\right\vert \\
&\leq &M\left( \frac{3\left( s^{2}+3s+4\right) }{(s+1)(s+2)(s+3)}\right) ^{%
\frac{1}{q}}\left[ \frac{(b-a)^{2}}{24}+\frac{1}{2}\left( x-\frac{a+b}{2}%
\right) ^{2}\right] .
\end{eqnarray*}%
This follows by Theorem \ref{teo3}, choosing $\left\vert f^{\prime \prime
}(x)\right\vert \leq M,\;M>0.$
\end{corollary}

\begin{corollary}
\label{cor7} With the assuptions in Corollary \ref{cor6}, one has the
mid-point inequality:%
\begin{equation*}
\left\vert \frac{1}{b-a}\int_{a}^{b}f(u)du-f\left( \frac{a+b}{2}\right)
\right\vert \leq M\left( \frac{3\left( s^{2}+3s+4\right) }{(s+1)(s+2)(s+3)}%
\right) ^{\frac{1}{q}}\frac{(b-a)^{2}}{24}.
\end{equation*}%
This follows by Corollary \ref{cor6}, choosing $x=\frac{a+b}{2}.$
\end{corollary}

\begin{remark}
\label{R2} If in Corollary \ref{cor7} we choose $s=1$ and $q=1,$ then we
have the following inequality:%
\begin{equation*}
\left\vert \frac{1}{b-a}\int_{a}^{b}f(u)du-f\left( \frac{a+b}{2}\right)
\right\vert \leq M\frac{(b-a)^{2}}{24}
\end{equation*}%
which is the inequality (\ref{e.1.3}).
\end{remark}

\begin{corollary}
\label{cor8} With the assumptions in Corollary \ref{cor6}, one has the
following perturbed trapezoid like inequality:%
\begin{multline*}
\left\vert \int_{a}^{b}f(u)du-\frac{\left( b-a\right) }{2}\left[ f\left(
a\right) +f\left( b\right) \right] +\frac{\left( b-a\right) ^{2}}{4}\left(
f^{\prime }\left( b\right) -f^{\prime }\left( a\right) \right) \right\vert \\
\leq \frac{(b-a)^{3}}{6}\left( \frac{3\left( s^{2}+3s+4\right) }{%
(s+1)(s+2)(s+3)}\right) ^{\frac{1}{q}}M.
\end{multline*}%
This follows using Corollary \ref{cor6} with $x=a$, $x=b$, adding the
results and using the triangle inequality for the modulus.
\end{corollary}

\begin{remark}
\label{R3} All of the above inequalities obviously hold for convex
functions. Simply choose $s=1$ in each of those results to get desired
results.
\end{remark}

The following result holds for $s-$concave.

\begin{theorem}
\label{teo4} Let $f:I\subset \lbrack 0,\infty )\rightarrow \mathbb{R}$ be a
twice differentiable function on $I^{\circ }$ such that $f^{\prime \prime
}\in L_{1}[a,b]$ where $a,b\in I$ with $a<b.$ If $\left\vert f^{\prime
\prime }\right\vert ^{q}$ is $s-$concave in the second sense on $[a,b]$ for
some fixed $s\in (0,1]$, $p,q>1$ and $\frac{1}{p}+\frac{1}{q}=1,$ then the
following inequality holds:%
\begin{eqnarray}
&&\left\vert \frac{1}{b-a}\int_{a}^{b}f(u)du-f(x)+\left( x-\frac{a+b}{2}%
\right) f^{\prime }(x)\right\vert  \label{e.2.9} \\
&\leq &\frac{2^{\left( s-1\right) /q}}{\left( 2p+1\right) ^{1/p}(b-a)}\left( 
\frac{\left( x-a\right) ^{3}\left\vert f^{\prime \prime }(\dfrac{x+a}{2}%
)\right\vert +\left( b-x\right) ^{3}\left\vert f^{\prime \prime }(\dfrac{b+x%
}{2})\right\vert }{2}\right)  \notag
\end{eqnarray}%
for each $x\in \lbrack a,b].$
\end{theorem}

\begin{proof}
Suppose that $q>1.$ From Lemma \ref{L.1} and using the H\"{o}lder
inequality, we have%
\begin{eqnarray*}
&&\left\vert \frac{1}{b-a}\int_{a}^{b}f(u)du-f(x)+\left( x-\frac{a+b}{2}%
\right) f^{\prime }(x)\right\vert \\
&\leq &\frac{(x-a)^{3}}{2(b-a)}\int_{0}^{1}t^{2}\left\vert f^{\prime \prime
}(tx+(1-t)a)\right\vert dt+\frac{(b-x)^{3}}{2(b-a)}\int_{0}^{1}t^{2}\left%
\vert f^{\prime \prime }(tx+(1-t)b)\right\vert dt \\
&\leq &\frac{(x-a)^{3}}{2(b-a)}\left( \int_{0}^{1}t^{2p}dt\right) ^{\frac{1}{%
p}}\left( \int_{0}^{1}\left\vert f^{\prime \prime }(tx+(1-t)a)\right\vert
^{q}dt\right) ^{\frac{1}{q}} \\
&&+\frac{(b-x)^{3}}{2(b-a)}\left( \int_{0}^{1}t^{2p}dt\right) ^{\frac{1}{p}%
}\left( \int_{0}^{1}\left\vert f^{\prime \prime }(tx+(1-t)b)\right\vert
^{q}dt\right) ^{\frac{1}{q}}.
\end{eqnarray*}%
Since $\left\vert f^{\prime \prime }\right\vert ^{q}$ is $s-$concave in the
second sense, using the inequality (\ref{e.1.1})%
\begin{equation}
\int_{0}^{1}\left\vert f^{\prime \prime }(tx+(1-t)a)\right\vert ^{q}dt\leq
2^{s-1}\left\vert f^{\prime \prime }(\frac{x+a}{2})\right\vert ^{q}
\label{e.2.10}
\end{equation}%
and 
\begin{equation}
\int_{0}^{1}\left\vert f^{\prime \prime }(tx+(1-t)b)\right\vert ^{q}dt\leq
2^{s-1}\left\vert f^{\prime \prime }(\frac{b+x}{2})\right\vert ^{q}.
\label{e.2.11}
\end{equation}%
A combination of (\ref{e.2.10}) and (\ref{e.2.11}) inequalities, we get 
\begin{eqnarray*}
&&\left\vert \frac{1}{b-a}\int_{a}^{b}f(u)du-f(x)+\left( x-\frac{a+b}{2}%
\right) f^{\prime }(x)\right\vert \\
&\leq &\frac{2^{\left( s-1\right) /q}}{\left( 2p+1\right) ^{1/p}(b-a)}\left( 
\frac{\left( x-a\right) ^{3}\left\vert f^{\prime \prime }(\dfrac{x+a}{2}%
)\right\vert +\left( b-x\right) ^{3}\left\vert f^{\prime \prime }(\dfrac{b+x%
}{2})\right\vert }{2}\right) .
\end{eqnarray*}%
This completes the proof.
\end{proof}

\begin{corollary}
\label{cor9} If in (\ref{e.2.9}), we choose $x=\frac{a+b}{2},$ then we have%
\begin{equation}
\left\vert \frac{1}{b-a}\int_{a}^{b}f(u)du-f\left( \frac{a+b}{2}\right)
\right\vert \leq \frac{2^{\left( s-1\right) /q}\left( b-a\right) ^{2}}{%
16\left( 2p+1\right) ^{1/p}}\left[ \left\vert f^{\prime \prime }(\dfrac{3a+b%
}{4})\right\vert +\left\vert f^{\prime \prime }(\dfrac{a+3b}{4})\right\vert %
\right] .  \label{e.2.12}
\end{equation}%
For instance, if $s=1,$ then we have 
\begin{equation*}
\left\vert \frac{1}{b-a}\int_{a}^{b}f(u)du-f\left( \frac{a+b}{2}\right)
\right\vert \leq \frac{\left( b-a\right) ^{2}}{16\left( 2p+1\right) ^{1/p}}%
\left[ \left\vert f^{\prime \prime }(\dfrac{3a+b}{4})\right\vert +\left\vert
f^{\prime \prime }(\dfrac{a+3b}{4})\right\vert \right] .
\end{equation*}
\end{corollary}

\section{Applications to Some Special Means}

Let us recall the following special means:

(1) The arithmetic mean: 
\begin{equation*}
A=A(x,y):=\frac{x+y}{2},\ x,y\geq 0;
\end{equation*}

(2) The Identric mean:%
\begin{equation*}
I=I\left( x,y\right) :=\left\{ 
\begin{array}{ccc}
x & \text{if} & x=y \\ 
&  &  \\ 
\frac{1}{e}\left( \frac{y^{y}}{x^{x}}\right) ^{\frac{1}{y-x}}\text{ } & 
\text{if} & x\neq y%
\end{array}%
\right. \text{, \ \ \ }x,y>0;
\end{equation*}

(c) The generalized log-mean:

\begin{equation*}
L_{p}=L_{p}(a,b):=\left\{ 
\begin{array}{ccc}
x & \text{if} & x=y \\ 
&  &  \\ 
\left[ \frac{y^{p+1}-x^{p+1}}{\left( p+1\right) \left( y-x\right) }\right] ^{%
\frac{1}{p}} & \text{if} & x\neq y%
\end{array}%
\right. \text{, \ \ \ }p\in \mathbb{R\diagdown }\left\{ -1,0\right\} ;\;x,y>0%
\text{.}
\end{equation*}

The following simple relationship is well known in the literature 
\begin{equation*}
I\leq A.
\end{equation*}%
It is known that $L_{p}$ is monotonic nondecreasing in $p\in \mathbb{R}$
with $L_{o}:=I.$

Now, using the results of Section 2, we give some applications to special
means of positive real numbers.

In \cite{hudzik}, the following example is given: Let $0<s<1$ and $u,v,w\in 
\mathbb{R}$. We define a function $f:\left[ 0,\infty \right) \rightarrow 
\mathbb{R}$%
\begin{equation*}
f(t)=\left\{ 
\begin{array}{lll}
u & if & t=0 \\ 
&  &  \\ 
vt^{s}+w & if & t>0.%
\end{array}%
\right.
\end{equation*}%
If $v\geq 0$ and $0\leq w\leq u$, then $f\in K_{s}^{2}$. Hence, for $u=w=0$, 
$v=1$, we have $f:\left[ 0,1\right] \rightarrow \left[ 0,1\right] $, $%
f(t)=t^{s},$ $f\in K_{s}^{2}.$

\begin{proposition}
\label{p1} Let $0<a<b$ and $s\in \left( 0,1\right) .$ Then we have the
results: 
\begin{eqnarray}
&&\left\vert L_{s}^{s}\left( a,b\right) -x^{s}+s\left( x-A\right)
x^{s-1}\right\vert  \label{EE1} \\
&\leq &3M\left( \frac{s^{2}+3s+4}{(s+1)(s+2)(s+3)}\right) \left[ \frac{1}{24}%
\left( b-a\right) ^{2}+\frac{1}{2}\left( x-A\right) ^{2}\right]  \notag \\
&\leq &M\frac{\left( b-a\right) ^{2}}{2}\left( \frac{s^{2}+3s+4}{%
(s+1)(s+2)(s+3)}\right)  \notag
\end{eqnarray}%
for all $x\in \left[ a,b\right] .$ If in the first inequality of (\ref{EE1})
we choose $x=A$, we get%
\begin{equation*}
\left\vert L_{s}^{s}\left( a,b\right) -A^{s}\right\vert \leq \frac{M\left(
b-a\right) ^{2}}{8}\left( \frac{s^{2}+3s+4}{(s+1)(s+2)(s+3)}\right) .
\end{equation*}
\end{proposition}

\begin{proof}
The inequality of (\ref{EE1}) follows from (\ref{e.2.6}) applied to the $s-$%
convex function in the second sense $f:\left[ 0,1\right] \rightarrow \left[
0,1\right] ,$ $f(x)=x^{s}.$ The details are omitted.
\end{proof}

\begin{proposition}
\label{p2} Let $0<a<b,$ $q>1$ and $s\in \left( 0,1\right) .$ Then we have
the results: 
\begin{eqnarray}
&&\left\vert L_{s}^{s}\left( a,b\right) -x^{s}+s\left( x-A\right)
x^{s-1}\right\vert  \label{EE2} \\
&\leq &\frac{3M}{\left( 2p+1\right) ^{1/p}}\left( \frac{2}{s+1}\right) ^{%
\frac{1}{q}}\left[ \frac{1}{24}\left( b-a\right) ^{2}+\frac{1}{2}\left(
x-A\right) ^{2}\right]  \notag
\end{eqnarray}%
for all $x\in \left[ a,b\right] .$ If in the first inequality of (\ref{EE2})
we choose $x=A$, we get%
\begin{equation*}
\left\vert L_{s}^{s}\left( a,b\right) -A^{s}\right\vert \leq \frac{M}{%
8\left( 2p+1\right) ^{1/p}}\left( \frac{2}{s+1}\right) ^{\frac{1}{q}}\left(
b-a\right) ^{2}.
\end{equation*}
\end{proposition}

\begin{proof}
The proof of (\ref{EE2}) is similar to that of (\ref{EE1}), using the
inequality (\ref{e.2.8}).
\end{proof}

\begin{proposition}
\label{p4} Let $0<a<b,$ $q>1$ and $s\in \left( 0,1\right) .$ Then we have
the results: 
\begin{eqnarray}
&&\left\vert L_{s}^{s}\left( a,b\right) -x^{s}+s\left( x-A\right)
x^{s-1}\right\vert  \label{EE3} \\
&\leq &M\left( \frac{3\left( s^{2}+3s+4\right) }{(s+1)(s+2)(s+3)}\right) ^{%
\frac{1}{q}}\left[ \frac{1}{24}\left( b-a\right) ^{2}+\frac{1}{2}\left(
x-A\right) ^{2}\right]  \notag
\end{eqnarray}%
for all $x\in \left[ a,b\right] .$ If in (\ref{EE3}) we choose $x=A$, we get%
\begin{equation*}
\left\vert L_{s}^{s}\left( a,b\right) -A^{s}\right\vert \leq \frac{M}{24}%
\left( \frac{3\left( s^{2}+3s+4\right) }{(s+1)(s+2)(s+3)}\right) ^{\frac{1}{q%
}}\left( b-a\right) ^{2}.
\end{equation*}
\end{proposition}

\begin{proof}
The proof is similar to that of Proposition \ref{p1}, using Corollary \ref%
{cor6}.
\end{proof}

\begin{proposition}
\label{p6} Let $0<a<b,$ $p>1$ and $s\in \left( 0,1\right) .$ Then we have
the result: 
\begin{equation*}
\left\vert \ln I-\ln A\right\vert \leq \frac{2^{\left( s-1\right) /q}\left(
b-a\right) ^{2}}{\left( 2p+1\right) ^{1/p}}\left[ -\frac{1}{\left(
3a+b\right) ^{2}}-\frac{1}{\left( a+3b\right) ^{2}}\right] .
\end{equation*}
\end{proposition}

\begin{proof}
The inequality follows from (\ref{e.2.12}) applied to the concave function
in the second sense $f:\left[ a,b\right] \rightarrow \mathbb{R}$, $f(x)=\ln
x $. The details are omitted.
\end{proof}

\end{document}